\documentclass{article}
\usepackage[centertags]{amsmath}
\usepackage{mathrsfs, amsthm, amssymb, amsfonts, color}
\usepackage{pifont}
\usepackage[top=2cm,bottom=2cm,left=2.3cm,right=2cm]{geometry}

\definecolor{couleurliens}{rgb}{.2,0.,.8} 
\usepackage[pdftex,colorlinks=true,pdfmenubar=true,bookmarks=true,bookmarksopen=true,pdftoolbar=true]{hyperref}
\hypersetup{%
	citecolor=magenta,
	filecolor=blue,
	linkcolor=couleurliens,
	urlcolor=cyan,
	pdftitle={$\alpha$-associated Metric On Rigged Null hypersurfaces},  
	pdfauthor={Hans FOTSING},                      
	pdfsubject={Hans Article}           
}%

\newtheorem{theo}{Theorem}[section]

\newtheorem{lem}{Lemma}[section]
\newtheorem{rem}{Remark}[section]
\newtheorem{pro}{Proposition}[section]

\newcommand{\R}{{\mathbb R}}

\newcommand{\Z}{{\mathbb Z}}

\newcommand{\Bm}{\overline{M}}
\newcommand{\bet}{\overline{\eta}}
\newcommand{\bg}{\overline{g}}
\newcommand{\bnab}{\overline{\nabla}}

\newcommand{\nabalp}{{\nabla_{\!\!\!\alpha}}}
\newcommand{\bnabalp}{{{\overline{\nabla}}_{\!\!\!\alpha}}}

\newcommand{\sn}{\stackrel{\star}{A}_{\xi}}

\title{$\alpha$-Associated Metrics on Rigged  Null Hypersurfaces}
\author{Ferdinand Ngakeu\footnote{Faculty of Science, University of Douala, Po. Box: 24157 Douala, Cameroon}\and Hans Fotsing Tetsing\footnote{African Institute for Mathematical Sciences, Po Box: 608 Limbe Southwest Cameroon}}

\begin{document}
\maketitle

\begin{abstract}
Let $x:M\to\Bm$ be  the canonical injection of a Null Hypersurface $(M,g)$ in a semi-Riemannian
manifold $(\overline{M},\bar g)$. A rigging for $M$ is a vector field $L$ defined on some open set of $\overline{M}$ containing $M$ such that $L_p\notin T_pM$ for each $p\in M$. Such a vector field induces a null rigging $N$. Let $\bar \eta$ be the 1-form which is $\bar g$-metrically equivalent to $N$ and
$\eta=x^\star\bar\eta$ its pull back on $M$. We introduce and study for a given  non vanishing function $\alpha$ on $M$ the  so-called $\alpha$-associated (semi-)Riemannian metric $ g_{\alpha}=g+\alpha\eta\otimes \eta$. For a closed rigging $N$ we give a constructive method to find an $\alpha$-associated metric whose Levi-Civita connection coincides with the  connection $\nabla$ induced on $M$ by the Levi-Civita connection $\overline{\nabla}$ of $\overline{M}$ and the null rigging $N$. We relate geometric objects of ${g}_{\alpha}$ to those of $g$ and $\overline{g}$.
As application, we show that given a null Monge hypersurface $M$ in $\R_q^{n+1},$ there  always exists a rigging  and an $\alpha$-associated metric whose Levi-Civita connection coincides with the induced connection  on $M$.
\end{abstract}

\underline{\bfseries keyword:}
Perturbation of metric, Monge hypersurface, Null hypersurface, screen distribution, Rigging vector field, Associated Metric

MSC[2010] 053C23, 53C25, 53C44, 53C50

\section{Introduction}\label{section1}

Let $(\Bm,\bg)$ be a proper semi-Riemannian manifold and $x:M\to\Bm$ be an embedded hypersurface of $\Bm$. The pull-back metric $g=x^\star\bg$ can be either degenerate or non-degenerate on $M$. When $g$ is non-degenerate, one says that $(M,g)$ is a semi-Riemannian hypersurface of $(\Bm,\bg)$ and if $g$ is degenerate then $(M,g)$ is said to be a null (or degenerate, or lightlike) hypersurface of $(\Bm,\bg)$. Since any semi-Riemannian hypersurface has a natural transversal vector field which is anywhere orthogonal to the hypersurface, there is a standard way to study such an hypersurface. Geometry tools of the ambient manifold $\Bm$ are projected orthogonally on $M$ and give new tools which can be used to study the geometry of the hypersurface.

For a null hypersurface the tangent bundle contains the orthogonal bundle, hence a null hypersurface cannot be studied the same way as a non degenerate hypersurface. One of the most used techniques   to study a null hypersurface $(M,g)$ in a semi-Riemannian manifold $(\overline{M},\bar g)$ is to fix arbitrarily on it a screen distribution $S(TM)$ and a null section $\xi\in \Gamma(TM)$. These choices fix locally a null transverse vector field $N$ which is orthogonal to the screen distribution and which verifies $\overline{g}(N,\xi)=1$ and leads to the decomposition of the tangent space $T\overline{M}$
(see for instance \cite{Atin-pseudo, Atin-blass, DB}). Instead of choosing a null section and a screen distribution independently, we can make only one arbitrary choice of a transverse vector field $L$ defined on an open neighborhood of $M$ in $\overline{M}$  and called the rigging for $M$. This choice induces a null section $\xi$ (called rigged vector field) a screen distribution and a null transverse vector field $N$. This second (rigging) technique has been introduced  in \cite{GO} and  also used in other works such as \cite{Atin-Fotsing, AFN, AGH, FN, Bo-cano, Bo-rig-tech}.

A null rigging $N$ for $M$ induces a family $(g_{\alpha})$ of non degenerate metrics on $M$ as follows: Let $\bar \eta$ be the $1-$form which is $\bar g$-metrically equivalent to $N$ (i.e.  $\bar \eta = \bar g ( N,.)$) and
$\eta$ the pull back of $\bar \eta$ on $M$ via the immersion $x^\star$. For a given  nowhere vanishing smooth function $\alpha$ on $M$ we define  the  so-called $\alpha$-associated (semi-)Riemannian metric on $M$ as  $ g_{\alpha}=g+\alpha\eta\otimes \eta$. When $\alpha=1$ the metric $ g_{1}=g+\eta\otimes \eta$ is usually called the induced metric or the rigged metric on $M$. It appears that for  functions $\alpha >0,$ one can choose a suitable rigging whose rigged metric coincides with $g_{\alpha}$. When the ambient space  $(\overline{M}, \bar g)$ is a Lorentzian manifold, the rigged metric is a pure Riemannian metric. It has been recently used in \cite{GO} to study Riemannian geometry of $M$ and also in \cite{AGH} to find new properties of the geometry of $M$.
Notice that when the null rigging $N$ is defined on  $\overline{M}$ it induces the so-called  perturbation
$ \bar g_{\alpha}=\bar g+\bar\alpha \bar\eta\otimes \bar \eta$ of the metric $\bar g$ whose restriction on $M$ gives also an associated metric. Such perturbations including those defined by spacelike or timelike vector fields at the place of null rigging have been considered in several works (see \cite{ST-HA} for $\alpha$-associated type and \cite{Bo-cano} for canonical variation $g_t=g+t\eta\otimes \eta$ where $t$ is constant). The Levi Civita connection of the $\alpha-$associated metric provides us with another connection $\nabalp$ on $M$. This does not coincide in general with the connection $\nabla$ induced on $M$ from the Levi-Civita connection $\overline{\nabla}$ of $\bar g$  by the projection along the chosen rigging for $M$. A necessary and sufficient condition to have this coincidence for the case $\alpha=1$ has been given in \cite{Atin-pseudo, Bo-rig-tech}.

  In this work we provide a constructive method to obtain rigging $N$ and metrics  $g_{\alpha}$ for which  the coincidence of $\nabla^{N}=\nabla$ and $\nabalp$ holds for functions $\alpha$  that are constant along leaves of the screen distribution. We precisely prove that if $(N;\alpha)$ is a solution for coincidence then for any $p\in \Z$ and any nowhere vanishing function $\phi$ on $M$ which is constant on leaves of the screen distribution, the couple
$(\phi^pN; \frac{\alpha}{\phi^{2^{^p}}})$ is also a solution for coincidence. We also relate Riemannian and  sectional  curvatures of $(M,\nabla)$ and those of $(M,\nabla_{\alpha})$. We give some applications of our formalism on null Monge  hypersurfaces  in $\R_1^{n+1}$.

This paper is organized as it follows: This first  Section is labeled as Introduction. In Section \ref{section2} we present the twisted metric or a perturbation of a semi-Riemannian metric along a null vector field. Section \ref{section3} is devoted to the general setup on null hypersurfaces and  new results on the $\alpha$-associated metric. Theorem \ref{theoconnection} gives necessary and sufficient condition for the $\alpha$-associated connection to coincide with the induced connection providing that $\alpha$ is constant along the leaves of the screen distribution. Section \ref{section4} is devoted to the computation of curvatures of the induced connection and the $\alpha$-associated connection. Finally in Section \ref{section5} we apply the formalism developed in the preceding sections to  null Monge hypersurfaces in
$\R_1^{n+1}$ by showing that  such  hypersurfaces  always admit suitable riggings and functions $\alpha$ such that $\nabalp
=\nabla$.


\section{Twisted metrics on a semi-Riemannian manifold}\label{section2}

Throughout this work, $(\Bm, \bg)$ is a $(n+1)-$dimensional semi-Riemannian manifold of index $q>0$, $\bnab$ and $\bar{R}$ will denote respectively the Levi-Civita connection and the Riemannian curvature of $\bg$. (Tools of the metric $\bg$ will be surmount with a line.) All manifolds are taken smooth and connected. Let $\Sigma$ be a $d-$dimensional manifold with $d\leq n+2$. If there exists an immersion $x:\Sigma\to\Bm$ then, $x(\Sigma)$ is said to be a $d-$dimensional {\bfseries immersed submanifold} of $\Bm$.  If moreover $x$ is injective one says that $x(\Sigma)$ is a $d-$dimensional {\bfseries submanifold} of $\Bm$. If in addition the inverse map $x^{-1}$ is a continue map from $x(\Sigma)$ to $\Sigma$, $x(\Sigma)$ is a $d-$dimensional {\bfseries embedded submanifold} of $\Bm$. When $x(\Sigma)$ is an embedded submanifold, one identify $\Sigma$ and $x(\Sigma)$. All submanifolds will be taken as embedded and through the identification, saying that $x:M\to\Bm$ is a submanifold will mean that there is an embedding $x:\Sigma\to\Bm$ such that $M=x(\Sigma)$. An hypersurface of $\Bm$ is a submanifold of $\Bm$ of dimension $d=n$. We will said that $x:(M,g)\to(\Bm,\bg)$ is an isometrically immersed submanifold when, $x:M\to\Bm$ is a submanifold of $\Bm$ and $g=x^\star\bg$. An isometrically immersed submanifold $x:(M,g)\to(\Bm,\bg)$ will said to be a {\bfseries non-degenerate submanifold} if $(M,g)$ is a semi-Riemannian manifold. Otherwise, one says that $(M,g)$ is a {\bfseries degenerate or null or lightlike submanifold}. This means that at each point $p\in M$ there exists a nonzero vector $u\in T_pM$ such that $g_p(u,v)=0$ for any $v\in T_pM$.

Let $N$ be a lightlike vector field globally defined on $\Bm$ and $\alpha$ be a nowhere vanishing smooth function on $\Bm$. We set $\bet$ to be the $1-$form $\bg-$metrically equivalent to $N$. Using $\bg$, we define the {\bfseries $\alpha-$twisted metric} on $\Bm$ as
\begin{equation}\label{bargalpha}
\bg_\alpha=\bg+\alpha\bet\otimes\bet.
\end{equation}
\begin{lem}
	The pair $(\Bm,\bg_\alpha)$ is a semi-Riemannian manifold.
\end{lem}

\begin{proof}
	Let $p\in\Bm$ and $u\in T_p\Bm$ such that $\bg_\alpha(u,v)=0$ for all $v\in T_p\Bm$. In particular, $\bg_\alpha(u,N_p)=0$ and hence, $\bet(u)=0$ since $N_p$ is a null vector. It follows that $\bg(u,v)=0$ for all $v\in T_p\Bm$ and then $u=0$ since $\bg$ is non-degenerate. This proves that $\bg_\alpha$ is non-degenerate on $\Bm$.
\end{proof}

Let $\bnabalp$ be the Levi-Civita connection of $\bg_\alpha$. The metrics $\bg$ and $\bg_\alpha$ are two semi-Riemannian metrics on $\Bm$. The following gives relationship between their Levi-Civita connections.
\begin{pro}
	The connections $\bnab$ and $\bnabalp$ are related by
	\begin{align}
	 \bnabalp_UV=\bnab_UV&+\frac12\left[\alpha\bet(U)\left(i_Vd\bet\right)^{\#\bg_\alpha}+\alpha\bet(V)\left(i_Ud\bet\right)^{\#\bg_\alpha}-\bet(U)\bet(V)d\alpha^{\#\bg_\alpha}\right]\nonumber\\
	&+\frac12\left[\alpha\left(L_N\bg\right)(U,V)+d\alpha(U)\bet(V)+d\alpha(V)\bet(U)\right]N,\label{related}
	\end{align}
	where $d\alpha^{\#\bg_\alpha}$ is the vector field $\bg_\alpha-$metrically equivalent to the $1-$form $d\alpha$, and $L_N\bg$ is the Lie derivative of $\bg$ along $N$.
\end{pro}
\begin{proof}
	Let us start by recalling the Koszul equation defining $\bnabalp$. For all sections $U,V,W$ of the tangent bundle $T\Bm$,
	\begin{align*}
	2\bg_\alpha(\bnabalp_UV,W)&=U\cdot\bg_\alpha(V,W)+V\cdot\bg_\alpha(W,U)-W\cdot\bg_\alpha(U,V)\\
	&+\bg_\alpha([U,V],W)-\bg_\alpha([V,W], U)+\bg_\alpha([W,U], V).
	\end{align*}
	Using (\ref{bargalpha}) and the fact that $\bnab$ is torsion-free and $\bg-$metric, the later equation leads to
	\begin{align*}
	2\bg_\alpha(\bnabalp_UV,W)&=\bg(\bnab_UV,W)+\bg(V,\bnab_UW)+\alpha U\cdot(\bet(V)\bet(W))+d\alpha(U)\bet(V)\bet(W)\\
	&+\bg(\bnab_VW,U)+\bg(W,\bnab_VU)+\alpha V\cdot(\bet(U)\bet(W))+d\alpha(V)\bet(U)\bet(W)\\
	&-\bg(\bnab_WU,V)-\bg(U,\bnab_WV)-\alpha W\cdot(\bet(U)\bet(V))-d\alpha(W)\bet(U)\bet(V)\\
	&+\bg(\bnab_UV-\bnab_VU,W)+\alpha\bet([U,V])\bet(W)-\bg(\bnab_VW-\bnab_WV,U)\\
	&-\alpha\bet([V,W])\bet(U)+\bg(\bnab_WU-\bnab_UW,V)+\alpha\bet([W,U])\bet(V)\\
	&=2\bg(\bnab_UV,W)+2\alpha\bet(\bnab_UV)\bet(W)+\alpha\left(L_N\bg\right)(U,V)\bet(W)+d\alpha(U)\bet(V)\bet(W)\\
	&+\alpha\bet(U)d\bet(V,W)+\alpha\bet(V)d\bet(U,W)+d\alpha(V)\bet(U)\bet(W)-d\alpha(W)\bet(U)\bet(V)\\
	&=2\bg_\alpha(\bnab_UV,W)+\alpha\left(L_N\bg\right)(U,V)\bg(N,W)+d\alpha(U)\bet(V)\bg(N,W)\\
	&+\alpha\bet(U)d\bet(V,W)+\alpha\bet(V)d\bet(U,W)+d\alpha(V)\bet(U)\bg(N,W)-d\alpha(W)\bet(U)\bet(V),
	\end{align*}
	and (\ref{related}) holds.
\end{proof}

\section{Null hypersurfaces}\label{section3}

\subsection{$\alpha-$Associated metric and $\alpha-$twisted metric}

Let $x:(M,g)\to(\Bm,\bg)$ be a null hypersurface of $(\Bm,\bg)$. A {\bfseries rigging for $M$} is a vector field $L$ defined on an open subset containing $M$ and such that for any $p\in M$, $L_p\notin T_pM$. One says that a rigging $L$ is a null rigging for $M$ when the restriction of $L$ on $M$ is lightlike. Therefore if $N$ is a null vector field on $\Bm$ anywhere transversal to $M$, then $N$ is a null rigging for $M$.
We now recall some basic tools necessary for studying a null hypersuface. For more details see \cite{FN,GO,DB}. Let $\xi$ be the associated rigged vector field, and $\eta=x^\star\bet$. Setting the screen distribution $\mathcal S(TM)=ker(\eta)$ and the transverse bundle $tr(TM)=span(N)$, the following decompositions hold
\begin{equation}\label{decomp}
T\Bm_{|M}=TM\oplus tr(TM)=\mathcal
S(TM)\oplus_{orth}\left(TM^\bot\oplus tr(TM)\right).
\end{equation}
Recall that $\xi$ is the unique section of the radical bundle $Rad(TM)=TM^\perp=\{X\in\Gamma(TM);~g(X,Y)=0, \forall Y\in\Gamma(TM)\}$ such that
\begin{equation}\label{normalization}
\bg(\xi, N)=1, \;\; \bg(N,N)=\bg(N,X)=0, \;\;\forall
X\in\Gamma(\mathcal S(TM)).
\end{equation}
Let $\nabla$ be the connection on $M$ induced from $\bnab$ through the projection along the transverse bundle $tr(TM)=span(N)$. When confusion is possible in reason of many riggings, we denote the induced connection by $\nabla^N$. For every section $X$ of $TM$, one has $\bg(\bnab_X\xi,\xi)=0$, which shows that $\bnab_X\xi\in\Gamma(TM)$. The Weingarten map is the endomorphism field
$$\begin{matrix}
\chi:&\Gamma(TM)&\to&\Gamma(TM)\\&X&\mapsto&\bnab_X\xi
\end{matrix}.$$

The Gauss-Weingarten equations of the immersion $x:(M,g)\to(\Bm,\bg)$ are given by
\begin{eqnarray}
\bnab_XY&=&\nabla_XY+B(X,Y)N,\label{geq1}\\
\nabla_XPY&=&\stackrel\star\nabla_XPY+C(X,PY)\xi,\label{geq2}\\
\bnab_XN&=&-A_NX+\tau(X)N, \label{geq3}\\
\nabla_X\xi&=&-\sn X-\tau(X)\xi, \label{geq4}
\end{eqnarray}
for any $X,Y\in\Gamma(TM)$, where $\stackrel\star\nabla$, denotes the connection on the screen distribution $S(TM)$ induced by $\nabla$ through the projection morphism $P$ of $\Gamma(TM)$ onto $\Gamma(\mathcal S(TM))$ along $\xi$. $B$ and $C$ are the local second fundamental forms of $M$ and $\mathcal S(TM)$ respectively, $A_N$ and $\stackrel\star A_\xi$ are the shape operators on $TM$ and $\mathcal S(TM)$ respectively, and the {\bfseries rotation $1-$form} $\tau$ is given by $\tau(X)=\bg(\bnab_XN,\xi)$.

Shape operators and second fundamental forms are related by
\begin{align}
B(X,Y)&=g(\sn X,Y)\label{baxi}\\
C(X,PY)&=g(A_NX,Y).\label{can}
\end{align}
Using (\ref{normalization}), (\ref{geq1}) and (\ref{baxi}), it is straightforward to show that $\sn$ is $g-$symmetric and $\sn(\xi)=0$. On the contrary, $A_N$ is not necessarily $g-$symmetric. However, $A_N$ is $g-$symmetric on the screen distribution, as a consequence of the following Lemma.
\begin{lem}
	For any sections $X,Y$ of the tangent bundle $TM$, one has
	$$g(A_NX,Y)-g(X,A_NY)=\tau(X)\eta(Y)-\tau(Y)\eta(X)-d\eta(X,Y).$$
\end{lem}
\begin{proof}
	Just compute $d\eta(X,Y)$ by using covariant derivative and Gauss-Weingarten equations.
\end{proof}
The mean curvatures of $M$ and $S(TM)$ are respectively given by
$$\stackrel{\star}{H}=\frac{1}{n}tr\left(\sn\right)~~\mbox{and}~~H=\frac{1}{n}tr(A_N).$$
A null hypersurface $M$ is said to be totally umbilical (resp. totally geodesic) if there exists a smooth function $\rho$ on $M$ such that at each point $x \in M $ and for all
$X, Y \in T_x M$, $B(x)(X, Y) = \rho(x)g(X, Y)$ (resp. $B$ vanishes identically on $M$ ).
This is equivalent to write respectively
$\sn = \rho P$ and $\sn = 0$. Notice that these are intrinsic notion on any null hypersurface in the way that total umbilicity and total geodesibility of $M$ are independent of the choice of rigging. Also, the screen distribution
$S(TM)$ is totally umbilical (resp. totally geodesic) if there exists a smooth function $\lambda$ on $M$
such that $C(x)(X, P Y ) = \lambda(x) g(X, Y )$ for all $X, Y \in T_xM$ (resp. $C = 0$), which is equivalent to $A_N = \lambda P$ (resp. $A_N = 0$). We said that the rigged null hypersurface $x:(M,g,N)\to(\Bm,\bg)$ (or the rigging $N$) is with conformal screen distribution when there exists a non-vanishing smooth function $\varphi$ on $M$ such that
$$A_N=\varphi\sn.$$
When the $1-$form $\eta$ is closed, we said that $(M,g,N)$ is a closed rigged null hypersurface.
\begin{lem}\cite{FN}\label{tauvanish}
	For any closed rigged null hypersurface with conformal screen distribution, the rotation $1-$form vanishes on the screen distribution.
\end{lem}

For sections $X,Y,Z,T$ of $TM$, the so-called Gauss-Codazzi equations of $(M,g,N)$ are given by
\begin{align}
\bg(\overline{R}(X,Y)Z,PT)&=g(R(X,Y)Z,PT)\nonumber\\\label{ge1}&+B(X,Z)C(Y,PT)-B(Y,Z)C(X,PT)\\
\bg(\overline{R}(X,Y)Z,N)&=\bg(R(X,Y)Z,N)\label{ge}\\
\bg(\overline{R}(X,Y)PZ,N)&=\left(\nabla_XC\right)(Y,PZ)-\left(\nabla_YC\right)(X,PZ)\nonumber\\&+C(X,PZ)\tau(Y)-C(Y,PZ)\tau(X),\label{ge2}\\
\bg(\overline{R}(X,Y)Z,\xi)&=\left(\nabla_XB\right)(Y,Z)-\left(\nabla_YB\right)(X,Z)\nonumber\\&+B(Y,Z)\tau(X)-B(X,Z)\tau(Y).\label{ge3}\\
\bg(\overline{R}(X,Y)\xi,N)&=C(Y,\sn\!X)-C(X,\sn\!Y)-d\tau(X,Y),\label{ge4}
\end{align}
where $\nabla_X C$ is defined by $\left(\nabla_XC\right)(Y,PZ)=X\cdot C(Y,PZ)-C(\nabla_XY,PZ)-C(Y,\stackrel{\star}{\nabla}_X\!PZ)$.

For $\alpha\in\mathcal{C}^\infty(\Bm)^\star$ a non-vanishing smooth function, the pullback (restriction) $x^\star \bg_\alpha$ of the twisted metric (\ref{galpha}) on $M$ is given by
\begin{equation}\label{galpha}
g_\alpha=g+\alpha\eta\otimes\eta.
\end{equation}
We call $g_\alpha$ the {\bfseries $\alpha-$associated metric} of $(M,g,N)$. It is well-known that $\xi$ is the vector field $g_1-$metrically equivalent to the $1-$form $\eta$. Notice that the pull back $x^*\alpha$ of $\alpha$ is simply denoted again by $\alpha$.
\begin{lem}
	The pair $(M,g_\alpha)$ is a semi-Riemannian manifold of index $$\nu_\alpha=q-\frac12(1+sign(\alpha))=\begin{cases}
	q-1 & \mbox{ if } \alpha>0\\ q & \mbox{ if } \alpha<0
	\end{cases}.$$
\end{lem}
\begin{proof}
	Let $x\in M$ and $u\in T_xM$ such that $g_\alpha(u,v)=0$ for all $v\in T_xM$. In particular, $0=g_\alpha(u,\xi_x)=\alpha(x)\eta_{x}(u)\Rightarrow\eta_{x}(u)=0$ since $\alpha(x)\neq0$. Thus $u\in S(TM)_x$. One then has $g(u,v)=0$ for all $v\in S(TM)_x$, and hence $u=0$ since $g$ is non-degenerate on the screen distribution $S(TM)$. Thus $(M,g_\alpha)$ is a semi-Riemannian manifold. For the index, just remark that $g$ is of index $q-1$ on $S(TM)$ and $g_\alpha(\xi,\xi)=\alpha$.
\end{proof}
We now know that $x_\alpha:(M,g_\alpha)\to(\Bm,\bg_\alpha)$ is a non-degenerate hypersurface of the semi-Riemannian manifold $(\Bm,\bg_\alpha)$. The Gauss map of the isometrical immersion $x_\alpha$ is given by
\begin{equation}\label{gaussmap}
\delta_\alpha=\sqrt{|\alpha|}N-\frac{sign(\alpha)}{\sqrt{|\alpha|}}\xi.
\end{equation}
It is nothing to check that $\bg_\alpha(\delta_\alpha,\delta_\alpha)=-sign(\alpha)$. It follows that $(\Bm,\bg_\alpha)$ is a semi-Riemannian manifold of index $q$, since $(M,g)$ is of index $\nu_\alpha=q-\frac12(1+sign(\alpha))$ and the magnitude of the Gauss map of the immersion $x_\alpha$ is $-sign(\alpha)$. For the end of this subsection, we assume that the rigging $N$ is closed, this means that the equivalent $1-$form $\bet$ is closed. It is easy to check that this is equivalent to say
\begin{equation}\label{closed}
\bg(\bnab_UN,V)=\bg(U,\bnab_VN),~~\forall U,V\in\Gamma(T\Bm).
\end{equation} Using (\ref{related}) one has
\begin{equation*}
\bnabalp_X\delta_\alpha=\bnab_X\delta_\alpha+\frac12[\alpha\left(L_N\bg\right)(X,\delta_\alpha)+d\alpha(X)\bet(\delta_\alpha)+d\alpha(\delta_\alpha)\eta(X)]N-\frac12\eta(X)\bet(\delta_\alpha)d\alpha^{\#\bg_\alpha}
\end{equation*}
Using (\ref{geq1})-(\ref{closed}) and by direct calculations, we have
\begin{align*}
\bnab_X\delta_\alpha=\frac{sign(\alpha)}{\sqrt{|\alpha|}}&\left[-\alpha A_NX+\sn X+\tau(X)\xi\right]+(X\cdot\sqrt{|\alpha|}+\sqrt{|\alpha|}\tau(X))N+\frac{d\alpha(X)}{2\alpha}\xi,\\
(L_N\bg)(X,N)&=0,~~~~~~~~~~d\alpha^{\#\bg_\alpha}=d\alpha^{\#g_\alpha}-sign(\alpha)d\alpha(\delta_\alpha)\delta_\alpha,\\
(L_N\bg)(X,\xi)&=2\tau(X),~~~~~\bet(\delta_\alpha)=-\frac{sign(\alpha)}{\sqrt{|\alpha|}}.
\end{align*}
Thus,
\begin{equation*}
\bnabalp_X\delta_\alpha=\frac{sign(\alpha)}{\sqrt{|\alpha|}}\left[-\alpha A_NX+\sn X+\tau(X)\xi+\frac{d\alpha(X)}{2\alpha}\xi+\frac{\eta(X)}{2\sqrt{|\alpha|}}
\left(\sqrt{|\alpha|}d\alpha^{\#g_\alpha}+d\alpha(\delta_\alpha)\xi\right)\right].
\end{equation*}
The shape operator of the immersion $x_\alpha$ is then given by
\begin{equation*}
A_{\delta_\alpha}=\frac{sign(\alpha)}{\sqrt{|\alpha|}}\left[\alpha A_NX-\sn X-\tau(X)\xi-\frac{d\alpha(X)}{2\alpha}\xi-\frac{\eta(X)}{2\sqrt{|\alpha|}}
\left(\sqrt{|\alpha|}d\alpha^{\#g_\alpha}+d\alpha(\delta_\alpha)\xi\right)\right].
\end{equation*}

If $\alpha$ is constant on each leaf of the screen distribution and the screen distribution is conformal with conformal factor $1/\alpha$ then, the shape operator of the isometrical immersion $x_\alpha$ is given by
\begin{equation*}
A_{\delta_\alpha}(X)=-\frac{sign(\alpha)}{2\sqrt{|\alpha|}}\eta(X)\left[2\tau(\xi)+\eta(d\alpha^{\#g_\alpha})+d\alpha(N)\right]\xi.
\end{equation*}
We then have the following result.
\begin{theo}\label{theo}
	Let $x:(M,g,N)\to(\Bm^{n+1},\bg)$ be a closed rigged null hypersurface with conformal screen distribution with conformal factor $1/\alpha$ constant on leaves of the screen distribution $\mathscr S(N)$. Then, the isometrical immersion $x_\alpha:(M,g_\alpha)\to(\Bm,\bg_\alpha)$ ($\bg_\alpha$ being defined by (\ref{galpha})), is a non-degenerate hypersurface with at most two principal curvature: $0$ with multiplicity ~$n-1$ and eigenvectors the sections of $\mathscr S(N)$, and $-\frac{sign(\alpha)}{2\sqrt{|\alpha|}}\left[2\tau(\xi)+\eta(d\alpha^{\#g_\alpha})+d\alpha(N)\right]$ with multiplicity $1$ and eigenvectors the sections of $Rad(TM)$.
\end{theo}

\subsection{Induced metric and $\alpha-$associated metric.}

In this subsection, we are going to relate some geometric objects of the $\alpha-$associated metric $g_\alpha$ with the ones of the induced metric $g$. From here on, $N$ is strictly a null rigging for $M$. Just to say that we don't impose to $N$ to be lightlike globally on $\Bm$, but  on $M$. Recall that $\nabalp$ is the Levi-Civita connection of the $\alpha-$associated semi-Riemannian manifold $(M,g_\alpha)$ and $\nabla$ is the connection on the rigged null hypersurface $x:(M,g,N)\to(\Bm,\bg)$ induced from $\bnab$ through the projection along $N$.
\begin{pro}
	The connections $\nabalp$ and $\nabla$ are related by
	\begin{align}
	 \nabalp_XY&=\nabla_XY-\frac1{2}\eta(X)\eta(Y)d\alpha^{\#g_\alpha}+\frac{\alpha}{2}\left[\eta(X)(i_Yd\eta)^{\#g_\alpha}+\eta(Y)(i_Xd\eta)^{\#g_\alpha}\right]\nonumber\\
	&+\frac1{2\alpha}\left[\alpha\left(L_N\bg\right)(X,Y)+2B(X,Y)+d\alpha(X)\eta(Y)+d\alpha(Y)\eta(X)\right]\xi\label{related2}
	\end{align}
	
\end{pro}

\begin{proof}
	Reasoning as in the proof of (\ref{related}), one has
	\begin{align*}
	2g_\alpha(\nabalp_XY,Z)
	&=\bg(\bnab_XY,Z)+\bg(Y,\bnab_XZ)+\alpha X\cdot(\eta(Y)\eta(Z))+d\alpha(X)\eta(Y)\eta(Z)\\
	&+\bg(\bnab_YZ,X)+\bg(Z,\bnab_YX)+\alpha Y\cdot(\eta(X)\eta(Z))+d\alpha(Y)\eta(X)\eta(Z)\\
	&-\bg(\bnab_ZX,Y)-\bg(X,\bnab_ZY)-\alpha Z\cdot(\eta(X)\eta(Y))-d\alpha(Z)\bet(X)\bet(Y)\\
	&+\bg(\bnab_XY-\bnab_YX,Z)+\alpha\eta(\bnab_XY-\bnab_YX)\eta(Z)-\bg(\bnab_YZ-\bnab_ZY,X)\\
	&-\alpha\eta([Y,Z])\eta(X)+\bg(\bnab_ZX-\bnab_XZ,Y)+\alpha\eta([Z,X])\eta(Y)\\
	&=2g_\alpha(\nabla_XY,Z)+2B(X,Y)+\alpha\left(L_N\bg\right)(X,Y)\eta(Z)+d\alpha(X)\eta(Y)\eta(Z)\\
	&+\alpha\eta(X)d\eta(Y,Z)+\alpha\eta(Y)d\eta(X,Z)+d\alpha(Y)\eta(X)\eta(Z)-d\alpha(Z)\eta(X)\eta(Y).
	\end{align*}
	From here, using the fact that
	\begin{equation}
	\alpha\eta(X)=g_\alpha(X,\xi)~~~\forall X\in\Gamma(TM),
	\end{equation}
	one obtains (\ref{related2}).
\end{proof}

From here on, we assume that the rigging $N$ is closed. Then using (\ref{closed}), (\ref{geq3}) and (\ref{can}) one has
$$(L_N\bg)(X,Y)=2\tau(X)\eta(Y)-2C(X,PY),$$
and equation (\ref{related2}) becomes
\begin{align}
\nabalp_XY&=\nabla_XY-\frac1{2}\eta(X)\eta(Y)d\alpha^{\#g_\alpha}\nonumber\\
&+\frac1{2\alpha}\left[2B(X,Y)-2\alpha C(X,PY)+2\alpha\tau(X)\eta(Y)+d\alpha(X)\eta(Y)+d\alpha(Y)\eta(X)\right]\xi.\label{related2proof}
\end{align}

From now on, We use the following range of indexes:
$$i,j=0,1,\ldots,n;~~~~~a,b=1,\ldots,n~~~~~k,l=2,\ldots,n$$
for summations (often with Einstein summation convention). For free indexes, we shall use
$$\beta,\gamma=1,\ldots,n.$$
Let $\left(E_1=\frac{1}{\sqrt{|\alpha|}}\xi, E_2,\ldots,E_n\right)$ be an $g_\alpha-$orthonormal frame field of $TM$ such that $\left(E_2,\ldots,E_n\right)$ is a frame field of $S(TM)$. The matrix of $g_\alpha$ in this frame is given by
$$g_\alpha=(g_\alpha(E_a,E_b)),$$
and we set $(g_\alpha^{ab})$ to be the inverse matrix. Note that, $g_\alpha^{ab}=\varepsilon^a\delta^{ab}$, with $\varepsilon^a:=\pm1$.
\begin{pro}
	One has:
	\begin{dingautolist}{192}
		\item for all $X,Y$ sections of $TM$, $(L_\xi g_\alpha)(X,Y)=-2B(X,Y)+\eta(X)\eta(Y)d\alpha(\xi)$;
		\item in particular, $div^{g_\alpha}(\xi)=\frac{
			1}{2|\alpha|}d\alpha(\xi)-n\!\stackrel{\star}{H}$;
		\item if $\xi$ is $g_\alpha-$Killing conformal (or $\alpha-$Killing) with conformal factor $\varphi$ then, $(M,g,N)$ is totally umbilical (or geodesic) with umbilical factor $\rho=-\frac12\varphi$.
	\end{dingautolist}
\end{pro}
\begin{proof}
	Since $\nabalp$ is the Levi-Civita connection of $g_\alpha$, one has
	\begin{equation}\label{lie}
	(L_\xi g_\alpha)(X,Y)=g_\alpha(\nabalp_X\xi,Y)+g(X,\nabalp_Y\xi).
	\end{equation}
	Using (\ref{related2proof}), the latter becomes
	$$(L_\xi g_\alpha)(X,Y)=g_\alpha(\nabla_X\xi,Y)+g_\alpha(X,\nabla_Y\xi)
+\eta(X)\eta(Y)d\alpha(\xi)+\alpha\left[\eta(X)\tau(Y)+\eta(Y)\tau(X)\right].$$
	From here, using (\ref{galpha}) and Gauss-Weingarten equations, the first item holds.
	By definition and using (\ref{lie}) one has
	 $$div^{g_\alpha}(\xi)=tr\left(\nabalp\xi\right)=\varepsilon^kg_\alpha\left(\nabalp_{E_k}\xi,E_k\right)=\frac12\varepsilon^k(L_\xi g_\alpha)(E_k,E_k).$$
	From here using the first item, one obtains the second item. For the last item, let us assume that $\xi$ is $g_\alpha-$conformal Killing with conformal factor $\varphi$. Then the first item says that for all $X,Y$ sections of the tangent bundle $TM$,
	\begin{equation}\label{proof}
	-2B(X,Y)+\eta(X)\eta(Y)d\alpha(\xi)=\varphi g(X,Y)+\alpha\varphi\eta(X)\eta(Y).
	\end{equation}
	Set $X=Y=\xi$ one finds $d\alpha(\xi)=\alpha\varphi$ and (\ref{proof}) becomes
	$$-2B(X,Y)=\varphi g(X,Y),$$
	which complete the proof.
\end{proof}
With the above proof one sees that when $\xi$ is $g_\alpha-$Killing, $\alpha$ is necessarily constant along integral lines of $\xi$. We have two (family of) connections on $M$ namely, the induce connection $\nabla$ and the $\alpha-$associated connection $\nabalp$. A natural question is to ask if the both connections can be the same. The following result gives a necessary and sufficient condition to have an affirmative answer.

\begin{theo}\label{theoconnection}
	Let $x:(M,g,N)\to(\Bm,\bg)$ be a closed rigged null hypersurface.
	\begin{dingautolist}{192}
		\item Let $\alpha$ be a nowhere vanishing function constant on each leaf of the screen distribution. Then, the induce connection is the Levi-Civita connection of the $\alpha-$associated metric if and only if
		\begin{equation}\label{equadifconnection1}
		\sn=\alpha A_N ~~\mbox{ and }~~ 2\alpha\tau(\xi)+d\alpha(\xi)=0.
		\end{equation}
		
		\item Let $\alpha$ be a nonzero real number. Then, the induce connection is the Levi-Civita connection of the $\alpha-$associated metric if and only if
		\begin{equation}\label{equadifconnection2}
		\sn=\alpha A_N ~~\mbox{ and }~~ \tau\equiv0.
		\end{equation}
	\end{dingautolist}
\end{theo}

\begin{proof}

	If $\alpha$ is constant along the leaves of the screen distribution then,
	$$d\alpha(X)=\eta(X)d\alpha(\xi)~~\mbox{and}~~\alpha d\alpha^{\#g_\alpha}=d\alpha(\xi)\xi,$$
	
	and, equation (\ref{related2proof}) becomes
	\begin{equation}\label{related2proof3}
	\nabalp_XY=\nabla_XY+\frac1{2\alpha}\left[2B(X,Y)-2\alpha C(X,PY)+2\alpha\tau(X)\eta(Y)+\eta(X)\eta(Y)d\alpha(\xi)\right]\xi
	\end{equation}
	Now, $\nabalp=\nabla$ if and only if
	\begin{equation}\label{related2proof2}
	2B(X,Y)-2\alpha C(X,PY)+2\alpha\tau(X)\eta(Y)+\eta(X)\eta(Y)d\alpha(\xi)=0.
	\end{equation}
	Replacing $X$ and $Y$ by $\xi$ in the latter, one obtains $d\alpha(\xi)+2\alpha\tau(\xi)=0$. The latter together with (\ref{related2proof2}) allow us to conclude that if $\alpha$ is constant along the screen distribution then (\ref{equadifconnection1}) holds.
	Now if $\alpha$ is constant on $M$ then the screen distribution is conformal and $\tau(\xi)=0$ which by the Lemma \ref{tauvanish} implies that $\tau$ identically vanishes. The converse is straightforward by using (\ref{related2proof3}).
\end{proof}

By using Theorem 4.1 in \cite{Atin-pseudo}, the proof of the following result is a straightforward computation.

\begin{pro}\label{coincidence-factor}
	Let $(M,g,N)\to(\Bm,\bg)$ be a rigged null hypersurface. If $\alpha$ is a function such that (\ref{equadifconnection1}) holds then, the same equations hold for any change of rigging $\widetilde{N}=\phi N$, with $\phi$ constant on each leaf of the screen distribution, and for $\tilde\alpha=\frac{\alpha}{\phi^2}$.
\end{pro}

This Proposition tells us that if  $(N;\alpha)$ is a solution for coincidence then for any non-vanishing function $\phi$ on $M$ which is constant on leaves of the screen distribution
the couple $(\phi N; \frac{\alpha}{(\phi)^2})$
is also solution and by induction for any $p\in \Z$, the couple  $(\phi^pN; \frac{\alpha}{(\phi)^{2^p}})$ is  a solution for coincidence.
On the other hand the induced connexions on $M$ from riggings $N$ and $\phi N$ coincide, this is $\nabla^{\phi N}= \nabla^{N}$, from which we deduce that if $\nabalp=\nabla^N$ then for any non-vanishing positive function $\psi$ constant along leaves of the screen distribution,
we have $\nabla_{\!\!\!\alpha\psi}=\nabla^N$. From this we can say that if the Levi Civita Connection of the induced metric $g_1$ coincides with the induced connection, then so does the Levi-Civita connection of any variation
$g_t$ of $g$  for $t\in\R^\star_+$.

\section{Curvatures of the $\alpha-$associated metric}\label{section4}

In this section, $x:(M,g,N)\to(\Bm,\bg)$ is a closed normalized null hypersurface of a semi-Riemannian manifold, and $\alpha$ is a non-vanishing function on $M$ constant on each leaf of the screen distribution. Let $X,Y,Z$ be sections of $TM$. We recall that the Riemannian curvature $R_\alpha$ of the $\alpha-$associated metric $g_\alpha$ is given by
\begin{equation}\label{ralph}
R_\alpha(X,Y)Z=\nabalp_X\nabalp_YZ-\nabalp_Y\nabalp_XZ-\nabalp_{[X,Y]}Z.
\end{equation}
It is straightforward to relate each of the three terms of the right hand side of the above relation with tools of the lightlike metric. Using equation (\ref{related2proof}) and Gauss-Weingarten equations one finds
\begin{align*}
\nabalp_X\nabalp_YZ&=\nabla_X\nabla_YZ-\left[\frac{1}{\alpha}B(Y,Z)-C(Y,PZ)+\tau(Y)\eta(Z)+\frac{1}{2\alpha}\eta(Y)\eta(Z)d\alpha(\xi)\right]\sn X\\
&+\left\{\frac{1}{\alpha}B(\nabla_YX,Z)-C(X,P\nabla_YZ)+\tau(X)\eta(\nabla_YZ)+\frac{1}{2\alpha}\eta(X)\eta(\nabla_YZ)d\alpha(\xi)\right.\\
&\qquad-\left[\frac{1}{\alpha}B(Y,Z)-C(Y,PZ)+\tau(Y)\eta(Z)+\frac{1}{2\alpha}\eta(Y)\eta(Z)d\alpha(\xi)\right]\tau(X)\\
&\qquad-\frac{d\alpha(X)}{\alpha^2}\eta(Y)\eta(Z)d\alpha(\xi)+\frac{1}{2\alpha}X\cdot(\eta(Y)\eta(Z))d\alpha(\xi)\\
&\qquad\left.-\frac{d\alpha(X)}{2\alpha^2}B(Y,Z)+\frac{1}{\alpha}X\cdot B(Y,Z)-X\cdot C(Y,PZ)+X\cdot\left(\tau(Y)\eta(Z)\right)\right\}\xi.
\end{align*}
Similarly, we express the two other terms of (\ref{ralph}) to obtain the following:
\begin{pro}\label{ralpha}
	Riemannian curvatures of the connections $\nabalp$ and $\nabla$ are related by
	\begin{align*}
	 R_\alpha(X,Y)Z&=R(X,Y)Z-\left[\frac{1}{\alpha}B(Y,Z)-C(Y,PZ)+\tau(Y)\eta(Z)+\frac{1}{2\alpha}\eta(Y)\eta(Z)d\alpha(\xi)\right]\sn X\\
	 &+\left[\frac{1}{\alpha}B(X,Z)-C(X,PZ)+\tau(X)\eta(Z)+\frac{1}{2\alpha}\eta(X)\eta(Z)d\alpha(\xi)\right]\sn\!\!Y+d\tau(X,Y)\eta(Z)\\
	 &+\left\{\frac{1}{\alpha}\left(\nabla_XB\right)(Y,Z)-\frac{1}{\alpha}\left(\nabla_YB\right)(X,Z)+\left(\nabla_YC\right)(X,PZ)-\left(\nabla_XC\right)(Y,PZ)\right.\\
	&\quad-\left[\frac{1}{\alpha}B(Y,Z)-2C(Y,PZ)\right]\tau(X)+\left[\frac{1}{\alpha}B(X,Z)-2C(X,PZ)\right]\tau(Y)\\
	&\quad\left.+\frac{d\alpha(\xi)}{2\alpha^2}\big[\eta(Y)\left(2B(X,Z)-\alpha C(X,PZ)\right)-\eta(X)\left(2B(Y,Z)-\alpha C(Y,PZ)\right)\big]\right\}\xi.
	\end{align*}
\end{pro}
Let $\mathcal{X}, \mathcal{Y}, \mathcal{Z}, \mathcal{T}$ be sections of the screen distribution. Using the above proposition one finds
\begin{align*}
g_\alpha(R_\alpha(X, Y)Z,\mathcal X)&=g(R(X, Y)Z,\mathcal X)+\left[\frac{1}{\alpha}B(X,Z)-C(X, PZ)+\tau(X)\eta(Z)+\frac{1}{2\alpha}\eta(X)\eta(Z)d\alpha(\xi)\right]B(Y,\mathcal X)\\
&-\left[\frac{1}{\alpha}B(Y,Z)-C(Y,PZ)+\tau(Y)\eta(Z)+\frac{1}{2\alpha}\eta(Y)\eta(Z)d\alpha(\xi)\right]B(X,\mathcal X).
\end{align*}
Now using equation (\ref{ge1}) this becomes
\begin{align}
g_\alpha(R_\alpha(X, Y)Z,\mathcal X)&=\bg(\overline R(X,Y)Z,\mathcal X)-B(X,Z)C(Y,\mathcal X)+B(Y,Z)C(X,\mathcal X)\nonumber\\
&+\left[\frac{1}{\alpha}B(X,Z)-C(X, PZ)+\tau(X)\eta(Z)+\frac{1}{2\alpha}\eta(X)\eta(Z)d\alpha(\xi)\right]B(Y,\mathcal X)\label{ralphaxyzt}\\
&-\left[\frac{1}{\alpha}B(Y,Z)-C(Y,PZ)+\tau(Y)\eta(Z)+\frac{1}{2\alpha}\eta(Y)\eta(Z)d\alpha(\xi)\right]B(X,\mathcal X).\nonumber
\end{align}
Also, using equations (\ref{ge}) (\ref{ge2}) and the above proposition one finds
\begin{align}
\bg(R_\alpha(\xi,\mathcal X)\mathcal Y,N)&=\frac{1}{\alpha}\left(\nabla_\xi B\right)(\mathcal X,\mathcal Y)-\frac{1}{\alpha}\left(\nabla_{\mathcal X} B\right)(\xi,\mathcal Y)-\left[\frac{1}{\alpha}B(\mathcal X,\mathcal Y)-C(\mathcal X,\mathcal Y)\right]\tau(\xi)\nonumber\\
&-\frac{d\alpha(\xi)}{2\alpha^2}\left[2B(\mathcal X, \mathcal Y)+\alpha C(\mathcal X,\mathcal Y)\right]-C(\xi,\mathcal Y)\tau(\mathcal X)\label{ralphaxixy}
\end{align}
This equation (\ref{ralphaxixy}) together with Gauss-Codazzi equation (\ref{ge3}) give
\begin{align}
\bg(R_\alpha(\xi,\mathcal X)\mathcal Y,N)&=\frac{1}{\alpha}\bg(\overline{R}(\xi,\mathcal X)\mathcal Y,\xi)-\left[\frac{2}{\alpha}B(\mathcal X,\mathcal Y)-C(\mathcal X,\mathcal Y)\right]\tau(\xi)\nonumber\\
&-\frac{d\alpha(\xi)}{2\alpha^2}\left[2B(\mathcal X, \mathcal Y)+\alpha C(\mathcal X,\mathcal Y)\right]-C(\xi,\mathcal Y)\tau(\mathcal X)\label{ralphaxy}
\end{align}

In the Proposition \ref{ralpha}, we have given relationship between Riemannian curvatures of the connections $\nabalp$ and $\nabla$. Since $\nabla$ is not a $g-$metric connection, the $(1,3)-$tensor $R$ does not have all Riemannian curvature symmetries and does not allow to define classical Ricci tensor. However if one defines a Ricci tensor as $Ric(X,Y)=tr(Z\mapsto R(Z,X)Y)$, this gives a non necessarily symmetric tensor and the definition of the scalar curvature becomes ambiguous. For this reason, we are going to relate Ricci tensor of $\nabalp$ with the one of $\bnab$ for sections of $TM$. In \cite{GO}, such a relationship was found for $\alpha=1$ and by assuming that $M$ is totally geodesic. We are going to relate this Ricci tensor for a function $\alpha$ constant on the leaves of the screen distribution and without total geodesibility condition. Let us start with sections of the screen distribution.

\begin{pro}
	For all $\mathcal X, \mathcal Y$ sections of the screen distribution, one has
	\begin{align}
	Ric_\alpha(\mathcal X,\mathcal Y)&=\overline{Ric}(\mathcal X,\mathcal Y)-\bg(\overline{R}(\xi,\mathcal X)\mathcal Y,N)-\bg(\overline{R}(\xi,\mathcal Y)\mathcal X,N)+\frac{1}{\alpha}\bg(\overline{R}(\xi,\mathcal X)\mathcal Y,\xi)-C(\xi,\mathcal Y)\tau(\mathcal X)\nonumber\\
	&+\frac{1}{\alpha}g(\sn\!\mathcal X,\sn\!\mathcal Y)-g(\sn\!X,A_N\mathcal Y)-g(A_N\mathcal X,\sn\!\mathcal Y)+nB(\mathcal X,\mathcal Y)\left(H-\frac{1}{\alpha}\stackrel{\star}{H}\right)\label{ricalpha}\\
	&+nC(\mathcal X,\mathcal Y)\stackrel{\star}{H}-\left[\frac{2}{\alpha}B(\mathcal X,\mathcal Y)-C(\mathcal X,\mathcal Y)\right]\tau(\xi)-\frac{d\alpha(\xi)}{2\alpha^2}\left[2B(\mathcal X, \mathcal Y)+\alpha C(\mathcal X,\mathcal Y)\right]\nonumber
	\end{align}
\end{pro}
\begin{proof}
	By definition,
	$$Ric_\alpha(\mathcal X,\mathcal Y)=tr(\mathcal Z\mapsto R_\alpha(\mathcal Z,\mathcal X)\mathcal Y)=\sum_{k=2}^n\varepsilon^kg(R_\alpha(E_k,\mathcal X)\mathcal Y,E_k)+\bg(R_\alpha(\xi,\mathcal X)\mathcal Y,N).$$
	We are going to compute each term of the latter. Using (\ref{ralphaxyzt}) one has
	\begin{align*}
	\varepsilon^kg(R_\alpha(E_k,\mathcal X)\mathcal Y,E_k)&=\varepsilon^k\bg(\overline{R}(E_k,\mathcal X)\mathcal Y,E_k)-B(A_N\mathcal X,\mathcal Y)+nB(\mathcal X,\mathcal Y)H\\&+\frac{1}{\alpha}B(\sn\!\mathcal X,\mathcal Y)-B(\mathcal X,A_N\mathcal Y)-n\left[\frac{1}{\alpha}B(\mathcal X,\mathcal Y)-C(\mathcal X,\mathcal Y)\right]\stackrel{\star}{H}.
	\end{align*}
	Again by definition,
	$$\overline{Ric}(\mathcal X,\mathcal Y)=\varepsilon^kg(\overline{R}(E_k,\mathcal X)\mathcal Y,E_k)+\bg(\overline{R}(\xi,\mathcal X)\mathcal Y,N)+\bg(\overline{R}(\xi,\mathcal Y)\mathcal X,N),$$
	where we have use the quasi orthonormal basis $\left(N,\xi,E_2,\ldots,E_n\right)$. Hence,
	\begin{align*}
	\varepsilon^kg(R_\alpha(E_k,\mathcal X)\mathcal Y,E_k)&=\overline{Ric}(\mathcal X,\mathcal Y)-\bg(\overline{R}(\xi,\mathcal X)\mathcal Y,N)-\bg(\overline{R}(\xi,\mathcal Y)\mathcal X,N)\\&-B(A_N\mathcal X,\mathcal Y)+nB(\mathcal X,\mathcal Y)H\\&+\frac{1}{\alpha}B(\sn\!\mathcal X,\mathcal Y)-B(\mathcal X,A_N\mathcal Y)-n\left[\frac{1}{\alpha}B(\mathcal X,\mathcal Y)-C(\mathcal X,\mathcal Y)\right]\stackrel{\star}{H}.
	\end{align*}
	Then, one obtains (\ref{ricalpha}) by summing the latter with (\ref{ralphaxy}).
\end{proof}

To complete the computation of the Ricci of all two sections of $TM$, it is left to compute $Ric_\alpha(\xi,\xi)$ and $Ric_\alpha(\xi,\mathcal X)$.
\begin{pro}\label{ricalphaxixi}
	For any function $\alpha$ constant on each leaf of the screen distribution of the closed normalized null hypersurface $(M,g,N)\to(\Bm,\bg)$, the following hold.
	\begin{dingautolist}{192}
		\item $Ric_\alpha(\xi,\xi)=\overline{Ric}(\xi,\xi)-n\left[\tau(\xi)+\frac{1}{2\alpha}d\alpha(\xi)\right]\stackrel{\star}{H}$.
		\item For any section $\mathcal{X}$ of $S(TM)$, $$Ric_\alpha(\xi,\mathcal{X})=\overline{Ric}(\xi,\mathcal{X})+d\tau(\xi,\mathcal X)+ng(A_N\xi,\mathcal X)\stackrel{\star}{H}.$$
	\end{dingautolist}
\end{pro}

\begin{proof}
	By definition,
	$Ric_\alpha(\xi,X)=\varepsilon^kg_\alpha(R_\alpha(E_k,\xi)X,E_k)$. Equation (\ref{ralphaxyzt}) gives
	\begin{align}
	 \varepsilon^kg_\alpha(R_\alpha(E_k,\xi)X,E_k)&=\varepsilon^k\bg(\overline{R}(E_k,\xi)X,E_k)-\varepsilon^kB(E_k,X)C(\xi,E_k)\nonumber\\
	&+\varepsilon^k\left[C(\xi,PX)-\tau(\xi)\eta(X)-\frac{1}{2\alpha}\eta(X)d\alpha(\xi)\right]B(E_k,E_k).\label{ralphaproof}
	\end{align}
	Replacing $X$ by $\xi$ and summing on $k$ one finds
	 $$Ric_\alpha(\xi,\xi)=\overline{Ric}(\xi,\xi)=\sum\varepsilon_k\bg(\overline{R}(E_k,\xi)\xi,E_k)-n\left[\tau(\xi)+\frac{1}{2\alpha}d\alpha(\xi)\right]\stackrel{\star}{H}$$
	Since $\overline{Ric}(\xi,\xi)=\sum\varepsilon_kg(\overline{R}(E_k,\xi)\xi,E_k)$, the first item holds. Now replacing $X$ by $\mathcal{X}$ in (\ref{ralphaproof}) and summing one finds,
	$$Ric_\alpha(\xi,\mathcal X)=\overline{Ric}(\xi,\mathcal X)-\bg(\overline{R}(\xi,\mathcal{X})\xi,N)+g(A_N\xi,\sn\!\mathcal X)+ng(A_N\xi,\mathcal X)\stackrel{\star}{H},$$
	since $\overline{Ric}(\xi,\mathcal{X})=\varepsilon^kg(\overline{R}(E_k,\xi)\mathcal{X},E_k)+\bg(\overline{R}(N,\xi)\mathcal{X},\xi)=\varepsilon^kg(\overline{R}(E_k,\xi)\mathcal{X},E_k)+\bg(\overline{R}(\xi,\mathcal{X})\xi,N)$. Then using  Gauss-Codazzi equation (\ref{ge4}), the second item follows.
\end{proof}
The following relates sectional curvatures of $\nabalp$ and $\bnab$. Recall that the sectional curvature of a plane $\Pi=span(X,Y)$ is given by
$$K_\alpha(\Pi)=\frac{g_\alpha(R_\alpha(X,Y)X,Y)}{g_\alpha(X,X)g_\alpha(Y,Y)-g_\alpha(X,Y)^2}.$$
By using equation (\ref{ralphaxyzt}), the proof of the following Proposition is a straightforward calculation.
\begin{pro}
	Let $\mathcal X$ and $\mathcal Y$ be two orthogonal sections of the screen structure. Let us consider the planes $\Pi_0=span(\xi,\mathcal X)$ and $\Pi=span(\mathcal X,\mathcal Y)$. Then,
	\begin{dingautolist}{192}
		\item $K_\alpha(\Pi_0)=\frac{1}{\alpha g(\mathcal X,\mathcal X)}\left[\overline{K}_\xi(\Pi_0)+\left[\tau(\xi)+\frac{1}{2\alpha}d\alpha(\xi)\right]B(\mathcal X,\mathcal X)\right]$;
		\item.\vspace{-.5cm}
		\begin{align*}
		K_\alpha(\Pi)&=\overline{K}(\Pi)+\frac{B(\mathcal X,\mathcal X)B(\mathcal Y,\mathcal Y)-B(\mathcal X,\mathcal Y)^2}{\alpha g(\mathcal X,\mathcal X)g(\mathcal Y,\mathcal Y)}\\
		&+\frac{2B(\mathcal X,\mathcal Y)C(\mathcal X,\mathcal Y)-B(\mathcal X,\mathcal X)C(\mathcal 6Y,\mathcal Y)-B(\mathcal Y,\mathcal Y)C(\mathcal X,\mathcal X)}{g(\mathcal X,\mathcal X)g(\mathcal Y,\mathcal Y)}.
		\end{align*}
	\end{dingautolist}
\end{pro}

Let us now relate scalar curvatures of $(M,g_\alpha)$ and $(\Bm,\bg)$.
\begin{theo}
	Let $(M,g,N)\to(\Bm,\bg)$ be a closed rigged null hypersurface of a semi-Riemannian manifold and $g_\alpha$ the semi-Riemannian metric on $M$ defined as in (\ref{galpha}). The sectional curvatures $s_\alpha$ and $\overline{s}$ of $(M,g)$ and $(\Bm,\bg)$ respectively, are related (on $M$) by
	\begin{align*}
	s_\alpha&=\overline{s}-4\overline{Ric}(\xi,N)+2\overline{K}(\xi,N)+\frac{2}{\alpha}\overline{Ric}(\xi,\xi)-2tr\left(\sn\circ A_N\right)+\frac{1}{\alpha}tr\left(\sn^2\right)\\
	 &+n^2\left(2H-\frac{1}{\alpha}\stackrel{\star}{H}\right)\stackrel{\star}{H}-\tau(A_N\xi)+n\left(H-\frac{3}{\alpha}\stackrel{\star}{H}\right)\tau(\xi)-\frac{n}{2\alpha^2}d\alpha(\xi)\left(H+3\stackrel{\star}{H}\right).
	\end{align*}
\end{theo}

\begin{proof}
	By definition,
	$$s_\alpha=g_\alpha^{aa}Ric_\alpha(E_a,E_a)=\varepsilon^kRic_\alpha(E_k,E_k)+\frac{1}{\alpha}Ric_\alpha(\xi,\xi).$$
	Let us compute each term of the latter. Replacing $\mathcal{X}$ and $\mathcal{Y}$ by $E_k$ in equation (\ref{ricalpha}) and summing on $k$ one obtains
	\begin{align}
	 \varepsilon^kRic_\alpha(E_k,E_k)&=\overline{s}-4\overline{Ric}(\xi,N)+2\overline{K}(\xi,N)+\frac{1}{\alpha}\overline{Ric}(\xi,\xi)-2tr\left(\sn\circ A_N\right)+\frac{1}{\alpha}tr\left(\sn^2\right)\nonumber\\
	 &+n^2\left(2H-\frac{1}{\alpha}\stackrel{\star}{H}\right)\stackrel{\star}{H}-\tau(A_N\xi)+n\left(H-\frac{2}{\alpha}\stackrel{\star}{H}\right)\tau(\xi)-\frac{n}{2\alpha^2}d\alpha(\xi)\left(H+2\stackrel{\star}{H}\right)\label{c}
	\end{align}
	The first item of Proposition \ref{ricalphaxixi} gives.
	\begin{equation}\label{d}
	 \frac{1}{\alpha}Ric_\alpha(\xi,\xi)=\frac{1}{\alpha}\overline{Ric}(\xi,\xi)-\frac{n}{\alpha}\left[\tau(\xi)+\frac{1}{2\alpha}d\alpha(\xi)\right]\stackrel{\star}{H}
	\end{equation}
	One obtains the announced result by summing (\ref{c}) and (\ref{d}).
\end{proof}

\section{Application on Monge null hypersurfaces of $\R^{n+1}_q$}\label{section5}

Let us set now $(\Bm,\bg)=\R^{n+1}_q$, the real semi-Euclidean space with its canonical metric
$$\bg=\varepsilon_i(dx^i)^2,$$
with Einstein's summation and where $(x^0,\ldots,x^{n})$ is the rectangular coordinate of $\R^{n+1}$ and we have set
$$\varepsilon^i=\varepsilon_i:=\begin{cases}
-1 & \mbox{ if } 0\leq i\leq q-1\\
+1 & \mbox{ if } q\leq i\leq n
\end{cases}.$$
Let $\mathcal{D}$ be an open subset of $\R^n_{q-1}$ and let $F:\mathcal{D}\to\R$ be a nowhere vanishing smooth function. Let us consider the immersion
\begin{equation}
x:\begin{matrix}
&\mathcal{D}&\to&\R^{n+1}_q&\\
&p=(u^1,\ldots,u^n)&\mapsto&x(p)=(x^0=F(p), x^1=u^1,\ldots,x^n=u^n)&
\end{matrix}.
\end{equation}
Then $M=x(\mathcal{D})$ is called a Monge hypersurface. It is nothing to see that a vector field $X=X^i\frac{\partial}{\partial
	x^i}$ (Einstein's summation) on $\R^{n+1}_q$ is tangent to $M$ if $X^0=X^aF'_{u^a}$. Then $\mathbf n=\dfrac{\partial}{\partial
	x^0}+\varepsilon^aF'_{u^a}\frac{\partial}{\partial
	x^a}$ is normal to $M$. Then the Monge hypersurface $M$  is a null hypersurface if and only if $\mathbf n$ is a  null vector field. This is equivalent to say
\begin{equation}\label{lightlike}
\varepsilon^a\left(F'_{u^a}\right)^2=||\nabla F||^2=1,
\end{equation}
where $\nabla F$ is the gradient of $F$ in the semi-Euclidean space $\R^n_{q-1}$. Then, taking
partial derivative of (\ref{lightlike}) with respect to $x^\beta$ leads to
\begin{equation}\label{lightlike1}
\varepsilon^aF'_{u^a}F''_{u^au^\beta}=0.
\end{equation}

\subsection{Generic $UCC-$normalization on a Monge null hypersurface}

Let us endo	wed the  Monge null  hypersurface $x:M\to\R^{n+1}_q$ with the
(physically and geometrically) relevant rigging
\begin{equation}\label{rigging2}
\mathscr{N}_{F}=\frac{1}{\sqrt2}\Big[-\frac{\partial}{\partial
	x^0}+\varepsilon^aF'_{u^a}\frac{\partial}{\partial
	x^a}\Big].
\end{equation}
The corresponding rigged vector field is then given by
\begin{equation}\label{rigging3}
\xi_{F}=\frac{1}{\sqrt2}\mathbf n=\frac{1}{\sqrt2}\Big[\frac{\partial}{\partial
	x^0}+\varepsilon^aF'_{u^a}\frac{\partial}{\partial
	x^a}\Big].
\end{equation}
We show below that this is a closed normalization with vanishing rotation $1-$form $\tau$ and conformal
screen distribution with unit conformal factor $\varphi=1$.  In fact, let us consider  the natural
(global) parametrization of $M$ given by
\begin{equation}
\begin{cases} x^0=F(u^1,...,u^{n}) &\\
x^\alpha=u^\alpha \qquad\qquad\qquad\qquad & (u^1,...,u^{n})\in\mathcal D\\
\alpha=1,...,n
\end{cases}.
\end{equation}
Then $\Gamma(TM)$ is spanned by $\{\frac{\partial}{\partial
	u^\beta}\}_\beta$ with
\begin{equation}\label{partialua}
\frac{\partial}{\partial u^\beta}=F'_{u^\beta}\frac{\partial}{\partial
	x^0}+\frac{\partial}{\partial x^\beta}.
\end{equation}
Now taking covariant derivative of $\mathbf{n}$ by the flat connection $\overline\nabla$ and
using (\ref{lightlike1}) one has
\begin{eqnarray}
\overline\nabla_\frac{\partial}{\partial
	u^\beta}\mathbf n&=&\varepsilon^aF''_{u^\beta u^a}\frac{\partial}{\partial
	x^a}\nonumber\\
&=&\varepsilon^aF''_{u^\beta u^a}\left(F'_{u^a}\frac{\partial}{\partial
	x^0}+\frac{\partial}{\partial
	x^a}\right)\nonumber\\
\overline\nabla_\frac{\partial}{\partial
	u^\alpha}\mathbf n&=&\varepsilon^aF''_{u^\beta u^a}\frac{\partial}{\partial{u^a}}.\label{tauvanishes}
\end{eqnarray}
Using again (\ref{lightlike1}) one has
$$\bg(\overline\nabla_\frac{\partial}{\partial
	u^\beta}\mathbf n, \mathscr N_F)=\varepsilon^aF'_{x^\beta}F''_{x^\beta x^a}.$$
Hence, $\overline\nabla_\frac{\partial}{\partial
	u^\beta}\mathbf n$ is a section of the screen distribution.
\begin{pro}
	Let $x:(M,g,\mathscr{N}_F)\to\R^{n+1}_q$ be a  Monge null
	hypersurface graph of a function $F$ endowed with the rigging $\mathscr{N}_F$ as in
	(\ref{rigging2}). Then  the following hold.
	\begin{enumerate}
		\item The rigging $\mathscr N_F$ is closed and the corresponding rotation $1-$form $\tau^{\mathscr{N}_F}$  vanishes identically.
		\item The screen distribution is conformal with $\varphi=1$ as
		conformal factor.
		\item The screen distribution is integrable with leaves the level sets of the fonction $F$.
		\item The induced connexion $\nabla$ coincides with  the Levi-Civita connexion
		of the associated metric $g_1$, i.e
		$$\nabla_{\!\!1}=\nabla.$$
		\item  In the natural basis  $\{\frac{\partial}{\partial u^a}\}_a$, the divergence
		(with respect to the induced connexion) of a vector field
		$X=X^a\frac{\partial}{\partial u^a}$ takes the form
		$$div X=\frac{\partial X^a}{\partial u^a}$$
		(as in usual Euclidean case).
	\end{enumerate}
\end{pro}

\begin{proof}
	Since $\bnab$ is the flat and the difference between both of the vectors $\mathscr{N}_F$, $\xi_F$ and $\frac{1}{\sqrt{2}}\mathbf{n}$ is a constant vector then,
	$$\bnab_\cdot\mathscr N_F=\bnab_\cdot\xi_F=\frac{1}{\sqrt{2}}\bnab_\cdot\mathbf{n}.$$
	Then by using (\ref{tauvanishes}) and (\ref{geq4}), $\tau^{\mathscr{N}_F}$
	identically vanishes and
	\begin{equation}
	A_{\mathscr{N}_F}\left(\frac{\partial}{\partial
		u^\alpha}\right)=\stackrel{\star}{A}_{\xi_F}\left(\frac{\partial}{\partial
		u^\beta}\right)=-\frac1{\sqrt2}\varepsilon^aF''_{u^\beta u^a}\frac{\partial}{\partial{u^a}}.
	\end{equation}
	Hence, $\stackrel{\star}{A}_{\xi_F}=A_{\mathscr{N}_F}$ which shows that the
	screen distribution is conformal with  conformal
	factor $\varphi=1$. The $1-$form $\eta$ is given by
	\begin{equation*}
	\eta=\sqrt{2}F'_{u^a}du^a.
	\end{equation*}
	Using Gauss Lemma it follows that
	\begin{equation*}
	d\eta=\sqrt{2}F''_{u^au^b}du^b\wedge du^a=\sqrt2\sum_{a\neq b}\left(F''_{u^au^b}-F''_{u^bu^a}\right)du^b\otimes du^a=0.
	\end{equation*}
	Which shows that the rigging $\mathscr N_F$ is closed. Then, the
	screen distribution is integrable. Let us show now that leaves of the screen distribution are really  the level
	sets of $F$. Let $c\in Im(F)$ be a regular value of $F$ and $M_c=F^{-1}(c)$ the
	$c-$level set  of $F$ in $\R^n_{q-1}$. Then, $\psi_c: M_c\to\R^{n}_{q-1}$ is a
	semi-Riemanniann hypersurface of the semi-Euclidean space $\R^n_{q-1}$ and the
	Gauss map is the gradient $\nabla F$ of $F$. We take $\psi_c$ to
	be the inclusion map and  $M_c$ is a subset of $\mathcal{D}$. We then have
	the following diagram
	\begin{eqnarray}
	M_c\stackrel{\psi_c}\hookrightarrow
	\mathcal{D}&\stackrel{\psi}\longrightarrow&
	M\stackrel i\hookrightarrow\R^{n+1}_q\label{diag}\\
	p(u^1,...,u^{n})&\mapsto&x(x^0=F(u^1,...,u^{n}),x^1=u^1,...,x^{n}=u^{n}).\nonumber
	\end{eqnarray}
	
	We denote by $\stackrel\circ\nabla$ and $\nabla_c$ the Levi-Civita
	connections of $\R^n_{q-1}$ and $M_c$ respectively.
	Taking the Jacobian matrix of $\psi$, it is easy to check that for
	any $X\in\Gamma(TM_c)$,
	$\psi_\star(\psi_{c\star}X)=\psi_\star(X)=(\langle X,\nabla
	F\rangle,X)=(0,X)$ and
	\begin{align*}
	\langle\psi_\star(X),\xi_F\rangle&=(1/\sqrt2)\langle
	(0,X),(1,\nabla F)\rangle=(1/\sqrt2)\left(-0+\langle X,\nabla F\rangle\right)=0,\\
	\langle\psi_\star(X),\mathscr{N}_{F}\rangle&=(1/\sqrt2)\langle
	(0,X),(-1,\nabla F)\rangle=(1/\sqrt2)\left(0+\langle X,\nabla
	F\rangle\right)=0.
	\end{align*}
	Thus the level sets $\psi(M_c)$ are  leaves of the screen
	distribution $\mathscr{S}(\mathscr{N}_F)$ of $M$ (endowed with the
	normalization (\ref{rigging2})).
	
	Since $\tau^{\mathscr{N}_F}$ identically vanishes and
	$\stackrel{\star}{A}_{\xi_F}=A_{\mathscr{N}_F}$, $\nabla$ is the
	Levi-Civita connexion of the (semi-Riemannian) associate metric $g_1$ (see
	Theorem 4.1 in \cite{Atin-pseudo}).
	Let $X=X^a\frac{\partial}{\partial u^a}$ be a section of $TM$:
	$$X=X^a\frac{\partial}{\partial u^a}=X^0\frac{\partial}{\partial
		x^1}+X^a\frac{\partial}{\partial x^a},$$ with $X^0=F'_{u^a}X^a$.
	We have,
	$$\overline\nabla_{\partial_{u^b}}X=\partial_{u^b}(X^0)\partial x^0+\partial_{u^b}(X^a)\partial_{x^a}.$$
	By using (\ref{geq1}) and (\ref{rigging2}) the left hand side of the above equation gives
	\begin{align*}
	\overline\nabla_{\partial_{u^b}}X&=\nabla_{\partial_{u^b}}X+B(\partial_{u^b},X)\mathscr{N}_F\\
	&=f^a\partial_{u^a}+B(\partial_{u^b},X)\mathscr{N}_F\\
	&=\left(F'_{u^a}f^a-\frac{1}{\sqrt2}B(\partial_{u^b},X)\right)\partial_{x^0}\\
	 &+\sum_{a=1}^{q-1}\left(f^a-F'_{u^a}\frac{1}{\sqrt2}B(\partial_{u^b},X)\right)\partial_{x^a}+\sum_{a=q}^{n}\left(f^a+F'_{uâ}\frac{1}{\sqrt2}B(\partial_{u^b},X)\right)\partial_{x^a}.
	\end{align*}
	After identification, one gets
	$$f^a=\begin{cases}
	\partial_{u^b}(X^a)+\frac{1}{\sqrt2}F'_{u^a}B(\partial_{u^b},X) & \mbox{ if } 1\leq a\leq q-1\\
	\partial_{u^b}(X^a)-\frac{1}{\sqrt2}F'_{u^a}B(\partial_{u^b},X) & \mbox{ if } q\leq a\leq n
	\end{cases}.$$ Hence,
	 $$\nabla_{\partial_{u^b}}X=\sum_{a=1}^{q-1}\left(\partial_{u^b}(X^a)+\frac{1}{\sqrt2}F'_{u^a}B(\partial_{u^b},X)\right)\partial_{u^a}+\sum_{a=q}^{n}\left(\partial_{u^b}(X^a)-\frac{1}{\sqrt2}F'_{u^a}B(\partial_{u^b},X)\right)\partial_{u^a}.$$
	The above relation together with equation (\ref{lightlike}) lead to
	\begin{align*}
	div X&=tr(\nabla X)\\
	 &=\sum_{a=1}^{q-1}\left(\partial_{u^a}(X^a)+\frac{1}{\sqrt2}F'_{u^a}B(\partial_{u^a},X)\right)+\sum_{a=q}^{n}\left(\partial_{u^a}(X^a)-\frac{1}{\sqrt2}F'_{u^a}B(\partial_{u^a},X)\right)\\
	 &=\partial_{u^a}(X^a)+\frac{1}{\sqrt2}\sum_{a=1}^{q-1}F'_{u^a}B(F'_{u^a}\partial_{x^0}+\partial_{x^a},X)-\frac{1}{\sqrt2}\sum_{a=q}^{n}F'_{u^a}B(F'_{u^a}\partial_{x^0}+\partial_{x^a},X)\\
	&=\partial_{u^a}(X^a)-B(\xi_F,X)\\
	&=\partial_{u^a}(X^a).
	\end{align*}
\end{proof}

Hence on any Monge null hypersurface, our rigging
$\mathscr{N}_{F}$ has so many good properties, the screen
distribution is integrable, the $1-$form $\tau$ identically
vanishes and
\begin{equation}
A_{\mathscr{N}_{F}}=\stackrel\star A_{\xi_F}
\end{equation}
On a Monge null hypersurface, the rigging (\ref{rigging2}) will be called the generic Unitary Conformaly Closed (UCC-)rigging, since its is closed and with a conformal screen with conformal factor $\varphi=1$. Recall that a hypersurface of a semi-Riemannian manifold it said to be totally geodesic when its shape operator identically vanishes.
The above proposition together with Theorem \ref{theo} give the following result.
\begin{theo}
	For any Monge null hypersurface $(M,g,\mathscr N_F)\to\R^{n+1}_q$ endowed with its generic UCC-rigging (\ref{rigging2}), the isometrical immersion $(M,g_1)\to(\R^{n+1}, \bg_1)$ into the twisted semi-Riemannian manifold $(\R^{n+1}, \bg_1)$ with the metric (\ref{galpha}) is a totally geodesic non-degenerate hypersurface.
\end{theo}

\subsection{A special Rigging on a Monge null hypersurface of $\R^{n+1}_q$}

Let us consider now for $x:M\to\R^{n+1}_q$ the rigging
\begin{equation}\label{rigging4}
\mathcal{N}_{F}=\frac{1}{2x^0}\Big[\frac{\partial}{\partial
	x^0}-\varepsilon^aF'_{u^a}\frac{\partial}{\partial
	x^a}\Big]
\end{equation}
with corresponding rigged vector field
\begin{equation}\label{rigging5}
\mathcal\xi_{F}=-x^0\mathbf n=x^0\Big[-\frac{\partial}{\partial
	x^0}-\varepsilon^aF'_{u^a}\frac{\partial}{\partial
	x^a}\Big].
\end{equation}
This two vector fields are defined on $\R^\star\times\mathcal{D}$ which is an open subset containing our Monge null hypersurface $M$. But, they are lightlike only along $M$. Since $\mathcal{N}_F$ is conformal to the generic UCC-rigging, the rigging $\mathcal{N}_F$ also has integrable screen distribution and corresponding leaves are level sets of the function $F$. Furthermore for this rigging,
\begin{equation}\label{bet}
\bet=-\frac{1}{2x^0}\Big[dx^0+F'_{u^a}dx^a\Big]
\end{equation}
and
\begin{equation}\label{eta}
\eta=-\frac{1}{x^0}F'_{u^a}du^a
\end{equation}
since $dx^0=F'_{u^a}du^a$. Let us set for this subsection $\alpha=2(x^0)^2$, which is constant along the leaves of the screen distribution. By a direct calculation one finds 
\begin{equation}\label{mongegalpha}
g_\alpha=\left[\varepsilon_a+(F'_{u^a})^2\right](du^a)^2+2\sum_{a<b}F'_{u^a}F'_{u^b}du^a\wedge du^b,
\end{equation}
where $2dx^i\wedge dx^j=dx^i\otimes dx^j+dx^j\otimes dx^i$. Note that, $g_\alpha$ is a semi-Riemannian metric of index $q-1$ on $M$, but since $N$ is lightlike only along $M$, the metric $\bg_\alpha$ is not necessary non-degenerate. The problem is to find integers $n$ and $q$ for which this metric is non-degenerate, for then be able to apply results of Section \ref{section3} to the Monge null hypersurface $M$ endowed with this rigging. For example by a calculation of determinant, one shows that for $n=3$ and $q=2$, this metric $\bg_\alpha$ is non-degenerate for any $F$.

Using (\ref{partialua}) one has
\begin{equation*}
\bnab_{\frac{\partial}{\partial u^a}}\mathcal{\xi}_F=-x^0\bnab_{\frac{\partial}{\partial u^a}}\mathbf n-\frac{\partial x^0}{\partial u^a}\mathbf n=-x^0\bnab_{\frac{\partial}{\partial u^a}}\mathbf n+\frac{F'_{u^a}}{x^0}\mathcal\xi_F.
\end{equation*}
This latter together with (\ref{geq4}) and (\ref{eta}) give
\begin{equation}
\sn=x^0\bnab_{\cdot}\mathbf n~~~~\mbox{ and }~~~~\tau=\eta.
\end{equation}
Also,
\begin{equation*}
\bnab_{\frac{\partial}{\partial u^\beta}}\mathcal{N}_F=-\frac{1}{2x^0}\bnab_{\frac{\partial}{\partial u^\beta}}\mathbf n-(2x^0)\frac{\partial~1/(2x^0)}{\partial u^\beta}\mathcal N_F=-\frac{1}{2x^0}\bnab_{\frac{\partial}{\partial u^\beta}}\mathbf n-\frac{F'_{u^\beta}}{x^0}\mathcal N_F.
\end{equation*}
Which allows we to find
\begin{equation}
A_N=\frac{1}{2x^0}\bnab_{\cdot}\mathbf n.
\end{equation}
Then,  the screen distribution is conformal with
\begin{equation}
\sn=2(x^0)^2A_N~~~\mbox{ and }~~~\tau=\eta.
\end{equation}
From here, it is easy to check that (\ref{equadifconnection1}) holds. By Theorem \ref{theoconnection}, the induced connection is the Levi-Civita connection of the $\alpha-$associated metric $g_\alpha$. Thus, $\nabla=\nabalp$, where $\alpha =2(x^0)^2$ .

\begin{rem}
	By Proposition \ref{coincidence-factor}  it is noteworthy that for all change of rigging $\widetilde{N}=\phi\mathcal{N}_F$ where $\phi>0$  is a  function of $x^0$, the Levi-Civita connection of the $\alpha\phi-$associated metric $g_{\alpha\phi}$ coincides with the induced connection $\nabla^{\mathcal{N}_F}$.
\end{rem}

\noindent\underline{\bfseries Acknowledgment:} The authors thank B. Olea for having pointed out some references such as [11] and given some suggestions to improve the manuscript.

{\bfseries Email addresses:} fngakeu@yahoo.fr, hans.fotsing@aims-cameroon.org

\end{document}